\documentclass[12pt]{article}
\usepackage[a4paper,margin=1in,footskip=0.25in]{geometry}

\usepackage{amsthm,amsmath,amssymb}
\usepackage{thm-restate}
\usepackage{graphicx}
\usepackage[svgnames]{xcolor}
\usepackage[backref=page,bookmarks,citecolor=blue,colorlinks]{hyperref}
\usepackage[nameinlink]{cleveref}
\usepackage{tikz}
\usetikzlibrary{shadings, calc}

\colorlet{triforcefilloutercolor}{Black}
\colorlet{triforcefillinnercolor}{Black}
\colorlet{triforceoutlineinnercolor}{Black}
\colorlet{triforceoutlineoutercolor}{Black}

\tikzset{%
	triforcefillshade/.style={%
		inner color=triforcefillinnercolor,%
		outer color=triforcefilloutercolor%
	},%
	triforceoutlineshade/.style={%
		inner color=triforceoutlineinnercolor,%
		outer color=triforceoutlineoutercolor%
	}%
}

\newcommand{\triforce}[1]{%
	\begin{tikzpicture}%
		\newdimen\triforcewidth%
		\newdimen\triforceheight%
		\triforcewidth=#1%
		\pgfmathparse{sqrt(3)}%
		\pgfmathsetlength{\triforceheight}{\pgfmathresult / 2 * \triforcewidth}%
		\foreach \x / \y in {0 / 0, 0.5\triforcewidth / 0, 0.25\triforcewidth / 0.5\triforceheight}%
		{%
			\shade[triforcefillshade, xshift=\x, yshift=\y]%
			(0, 0)  -- +(.5\triforcewidth, 0) -- +(60:.5\triforcewidth) -- cycle;%
			\shade[triforceoutlineshade, xshift=\x, yshift=\y]%
			(0, 0)  -- +(.5\triforcewidth, 0) -- +(60:.5\triforcewidth) -- cycle%
			(30:.0175\triforcewidth) -- ($(60:.5\triforcewidth) + (-90:.0175\triforcewidth)$) -- ($(0.5\triforcewidth, 0) + (150:.0175\triforcewidth)$) -- cycle;%
		}%
	\end{tikzpicture}%
}

\newcommand{\Val}{\mathrm{val}}

\theoremstyle{definition}
\newtheorem{definition}{Definition}[section]

\newtheorem{question}[definition]{Question}

\newtheorem{remark}[definition]{Remark}
\newtheorem*{definition*}{Definition}
\newtheorem*{example*}{Example}
\newtheorem*{exercise*}{Exercise}
\newtheorem*{notation*}{Notation}
\newtheorem*{problem*}{Problem}
\newtheorem*{remark*}{Remark}

\theoremstyle{plain}

\newtheorem{conjecture}[definition]{Conjecture}
\newtheorem{corollary}[definition]{Corollary}

\newtheorem{proposition}[definition]{Proposition}
\newtheorem{theorem}[definition]{Theorem}
\newtheorem*{claim*}{Claim}
\newtheorem*{conjecture*}{Conjecture}
\newtheorem*{corollary*}{Corollary}
\newtheorem*{lemma*}{Lemma}
\newtheorem*{proposition*}{Proposition}
\newtheorem*{theorem*}{Theorem}

\newcommand\Z{{\mathbb{Z}}}
\newcommand\N{{\mathbb{N}}}

\newcommand\E{{\mathbb{E}}}

\title{On generalized corners and matrix multiplication}
\author{Kevin Pratt}

\begin{document}
\maketitle

\begin{abstract}
Suppose that $S \subseteq [n]^2$ contains no three points of the form $(x,y), (x,y+\delta), (x+\delta,y')$, where $\delta \neq 0$. How big can $S$ be? Trivially, $n \le |S| \le n^2$. Slight improvements on these bounds are obtained from  Shkredov's upper bound for the corners problem \cite{shkredov2006generalization}, which shows that $|S| \le O(n^2/(\log \log n)^c)$ for some small $c > 0$, and a construction due to Petrov \cite{petrovanswer}, which shows that $|S| \ge \Omega(n \log n/\sqrt{\log \log n})$.

Could it be that for all $\varepsilon > 0$, $|S| \le O(n^{1+\varepsilon})$? We show that if so, this would rule out obtaining $\omega = 2$ using a large family of abelian groups in the group--theoretic framework of \cite{cu2003,cohn2005group} (which is known to capture the best bounds on $\omega$ to date), for which no barriers are currently known. Furthermore, an upper bound of $O(n^{4/3 - \varepsilon})$ for any fixed $\varepsilon > 0$ would rule out a conjectured approach to obtain $\omega = 2$ of \cite{cohn2005group}. Along the way, we encounter several problems that have much stronger constraints and that would already have these implications.
\end{abstract}

\section{Introduction}

The exponent of matrix multiplication $\omega$ is the smallest number such that for any $\varepsilon > 0$, there exists an algorithm for multiplying $n \times n$ matrices using $O(n^{\omega + \varepsilon})$ arithmetic operations. Since Strassen's initial discovery that $\omega < 3$ \cite{strassen1969gaussian}, there has been much work on understanding this fundamental constant, with the end goal being the determination of whether or not $\omega = 2$. It is currently known that $2 \le \omega <2.3716$ \cite{williams2023new}.

The best upper bounds on $\omega$ obtained since 1987 \cite{strassen1987relative} can be understood as solutions to the following hypergraph packing problem. Let $M_n$ be the \emph{matrix multiplication hypergraph}, the tripartite 3-uniform hypergraph with parts $X_1 = X_2 = X_3 = [n]^2$, and where $((i,j), (k,l),(m,n)) \in X_1 \times X_2 \times X_3$ is a hyperedge if and only if $j=k, l=m, n=i$. Given an abelian group $G$, let $X_G$ be its ``addition hypergraph" with vertex sets $G \sqcup G \sqcup G$, and where $(a_1,a_2,a_3) \in G \times G \times G$ is a hyperedge exactly when $a_1+a_2+a_3=0$. Suppose that $X_G$ contains $k$ disjoint induced copies of $M_n$. Then
\begin{equation}\label{eqn:stppbd}
\omega < \log_n(|G|/k).
\end{equation}
Phrased in terms of the group--theoretic approach proposed by Cohn and Umans \cite{cu2003} and further developed by Cohn, Kleinberg, Szegedy, and Umans \cite{cohn2005group}, this is equivalent to proving upper bounds on $\omega$ via \emph{simultaneous triple product property} (STPP) constructions in abelian groups. The above inequality was established in \cite[Theorem 5.5]{cohn2005group}. It can be also be deduced via the \emph{asymptotic sum inequality} of \cite{schonhage1981partial}.

From this perspective, the best bounds on $\omega$ to date are obtained by taking $G$ to be a large power of a cyclic group --- specifically, $\Z_7^\ell$ with $\ell \to \infty$. However, in \cite{blasiak2017cap} ideas related to the resolution of the cap-set problem in additive combinatorics \cite{ellenberg2017large} were used to show that one cannot obtain $\omega = 2$ using groups of \emph{bounded exponent} --- such as $\Z_7^\ell$ --- via this approach. This obstruction is due to the fact that when $G$ has bounded exponent, there is power-savings over the trivial upper bound on the size of the largest induced matching in $X_G$ (also called a \emph{3-matching} \cite{sawin2018bounds}, or a \emph{tricolored sum-free set} \cite{blasiak2017cap}). For example, when $G = \mathbb{Z}_7^\ell$ the largest induced matching has size at most $O(6.16^\ell)$. On the other hand, $M_n$ contains an induced matching of size $n^{2-o(1)}$: if we identify vertices in $M_n$ with edges in the complete tripartite graph $K_{n,n,n}$, an induced matching in $M_n$ corresponds to a tripartite graph on at most $3n$ vertices where every edge is contained in a unique triangle, and the number of vertices in the induced matching equals the number of edges in this graph. A well-known construction in extremal combinatorics yields such a graph with $n^{2-o(1)}$ edges (see  \cite[Corollary 2.5.2]{zhao2023graph})\footnote{This is the Rusza-Szemer\'edi problem. The equivalence between induced matchings in $M_n$ and this problem was independently noted in \cite{alman2023matrix}.} and hence $M_n$ contains an induced matching of size $n^{2-o(1)}$. Modulo minor details, the claimed barrier then follows, as an efficient packing of copies of $M_n$ into $X_G$ would imply the existence of a large induced matching in $X_G$, a contradiction.\footnote{The techniques involved in the resolution of the cap--set problem (in particular, slice rank) actually give stronger ``tensor analogues" of this barrier; see \cite{christandl2018barriers, alman2021limits}.}

This is the only obstruction to obtaining $\omega = 2$ via the use of \Cref{eqn:stppbd} that we are aware of. Unfortunately,\footnote{Or fortunately, for the optimist.} this barrier says nothing about the viability of general abelian groups, as their addition hypergraphs may contain large induced matchings. For example, if $A$ is a 3-term arithmetic progression free (hereon abbreviated to 3AP-free) subset of $G$, then the subsets $A,A,-2A$ of the vertex sets of $X_G$ induce a matching of size $|A|$. Hence this barrier cannot apply to any group containing a 3AP-free subset of size $|G|^{1-o(1)}$, such as $\mathbb{Z}_n$ \cite{behrend1946sets}. Could one achieve $\omega = 2$ using cyclic groups, or perhaps products of cyclic groups of growing orders?

In this paper we identify problems in additive combinatorics whose answer we conjecture would rule out obtaining $\omega = 2$ using a large family of abelian groups for which the induced matching barrier is irrelevant. This family includes abelian groups with a bounded number of direct factors --- the ``opposite" condition of that of having bounded exponent. These problems have not been studied before as far as we are aware. Aside from their connections to fast matrix multiplication, we find them intrinsically interesting. We now discuss the simplest-to-state such problem.

\subsection{A skew corners problem}
The \emph{corners problem} in additive combinatorics asks for the size of the largest subset of $[n]^2$ containing no three points of the form
\[(x, y), (x,y+\delta),(x+\delta,y)\]
where $\delta \neq 0$. Ajtai and Szemer\'edi \cite{ajtai1974sets} settled this problem up to factors of $n^{o(1)}$ by proving an upper bound of $o(n^2)$ and a lower bound of $n^{2-o(1)}$. This problem is significant as it was the first multidimensional case of Szemer\'edi's theorem to be established, and for its application to the number-on-forehead model in communication complexity \cite{chandra1983multi}.

 Here is a subtle strengthening of the condition of the corners problem for which we know essentially nothing:

\begin{question}\label{quest:skewtris}
What is the size of the largest $S \subseteq [n]^2$ which does not contain three points of the form
\[(x,y), (x, y+\delta), (x + \delta, y')\] with $\delta \neq 0$?
\end{question}
That is, not only must $S$ avoid all corners, but given any two points in $S$ lying on the same vertical line, the \emph{entire vertical line} passing through the third point that would form a corner with these two points must be absent from $S$! Naturally, we call such a set of points \emph{skew corner-free}. See \Cref{fig:skew10} for an example of such a set.

\begin{figure}
	\begin{center}
	\includegraphics[scale=0.5]{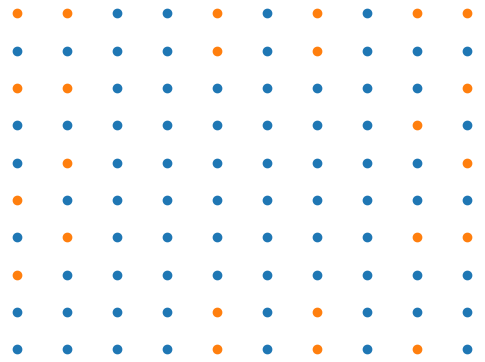}
		\caption{The orange points form a skew corner-free subset of $[10] \times [10]$ of size $24$. This is largest possible.}\label{fig:skew10}
	\end{center}
\end{figure}

Note that there is a trivial lower bound of $n$, obtained by taking $S$ to be all points lying on a vertical or horizontal line. We conjecture that this is almost optimal:

\begin{conjecture}\label{conj:skew}
Fix any $\varepsilon > 0$. If $S$ is skew corner-free, then $|S| \le O(n^{1+\varepsilon})$.
\end{conjecture}

A construction due to Petrov \cite{petrovanswer} (\Cref{prop:nontrivgrid}) shows that one can have $|S| \ge \Omega(n \log n / \sqrt{\log \log n})$. On the other hand, the best upper bound we know is $O(n^2/(\log \log n)^{ 0.0137 \cdots})$, which follows immediately from Shkredov's upper bound on the corners problem \cite{shkredov2006generalization}.

Two of the main results of this paper are the following.

\begin{restatable}{theorem}{main1}
	\label{thm:main1}
If \Cref{conj:skew} is true, then one cannot obtain $\omega = 2$ via STPP constructions in the family of groups $\mathbb{Z}_q^\ell$, where $q$ is a prime power.
\end{restatable}

Furthermore, a weakening of \Cref{conj:skew} would rule out obtaining $\omega = 2$ using a specific type of STPP construction in arbitrary abelian groups. In \cite{cohn2005group}, it was conjectured that this type of construction can be used to obtain $\omega = 2$.

\begin{restatable}{theorem}{main2}
	\label{thm:main2}
If the largest skew corner-free subset of $[n]^2$ has size $O(n^{4/3 - \varepsilon})$ for some $\varepsilon > 0$, then \cite[Conjecture 4.7]{cohn2005group} is false.
\end{restatable}

In fact, seemingly much weaker conjectures than \Cref{conj:skew} would already have these implications. The weakest conjecture we make is the following. Let $\Delta_n$ be a triangular array of $n(n+1)/2$ points. Suppose that we delete from $\Delta_n$ sets of points lying on lines parallel to the sides of this array, such that the remaining set of points does not contain any equilateral trapezoid with sides parallel to the sides of the array (see \Cref{fig:traps}). For example, we might delete all lines in one direction but one. Then, what is the maximum number of points that can remain? By our example, one can achieve at least $n$. We conjecture that this is essentially optimal (\Cref{conj:val}). Another condition we introduce, which is intermediate between this and being skew-corner free, is that of a skew corner-free subset of a triangular grid (see \Cref{fig:symm}).

\begin{figure}
	\begin{center}
	\rotatebox[origin=c]{180}{\includegraphics[scale=0.75]{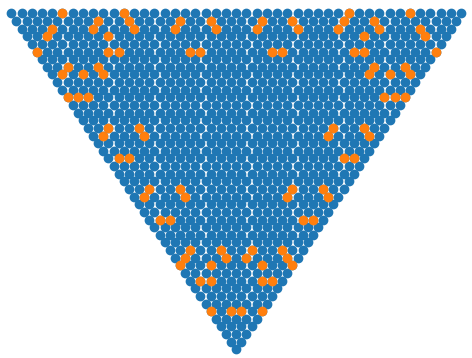}}
		\caption{The $90$ orange points form a skew corner-free subset of the triangular grid $\Delta_{45}$ (\Cref{def:skewcorners}): for any two orange points on the same line parallel to one of the sides of the grid, the line parallel to this side and passing through a third point that would form an equilateral triangle with these two points contains no orange points. This is largest-possible among subsets of $\Delta_{45}$ that are symmetric under the $S_3$ action on $\Delta_n$.}\label{fig:symm}
	\end{center}
\end{figure}

\subsection{Paper overview}
In \Cref{sec:bg} we review the group--theoretic approach of \cite{cu2003,cohn2005group}. In \Cref{sec:trirem} we record a very weak lower bound for this approach, which follows easily from the removal lemma in groups of \cite{serra2009combinatorial}. This lower bound becomes much stronger in $\mathbb{Z}_q^\ell$ (\Cref{cor:removaltight}), thanks to the improved bounds on the removal lemma of \cite{fox2017tight}, and we make later use of this fact.

In \Cref{sec:hyper} we note that the matrix multiplication hypergraph $M_n$ is an extremal solution to a certain forbidden hypergraph problem. This was our motivating observation. We define the ``value" of a group, $\Val(G)$, which captures this forbidden hypergraph problem in a group--theoretic context.  This quantity equals the maximum number of triangles in an induced subhypergraph of $X_G$ that does not contain the triforce hypergraph or a cycle of 4 triangles (see \Cref{fig:forbid}). This can also be expressed in terms of the group operation slightly awkwardly (\Cref{def:trap}). The trivial bounds are that $|G| \le \Val(G) \le |G|^{3/2}$; using the removal lemma of \cite{serra2009combinatorial}, the upper bound can be improved to $o(|G|^{3/2})$ (\Cref{prop:val2}). STPP constructions yield lower bounds on the quantity $\Val(G)$ (\Cref{prop:stppval}), so  ultimately it is upper bounds on $\Val(G)$ that we are interested in as a means towards barriers. The quantity $\Val(G)$ is super-multiplicative under direct product (\Cref{prop:tensorval}), which is one reason why power-improvements over the trivial bound seem to be easier to obtain in direct products of groups.

We then focus on the case of abelian groups in \Cref{sec:zn}. We show that a bound of $\omega = 2$ using the family of groups $\Z_q^\ell$ would imply that $\Val(\Z_n) \ge \Omega(n^{1+c})$ for some $c>0$ \Cref{thm:polyval}. We also show that a proof of $\omega = 2$ via \emph{simultaneous double product property} constructions \cite{cohn2005group} in any family of abelian groups would imply that $\Val(\Z_n) \ge \Omega(n^{4/3 - \varepsilon})$ for any given $\varepsilon>0$ (\Cref{thm:twofam43}). We thank Chris Umans for mentioning a related fact to us, which motivated this result. We then relate $\Val(\Z_n)$ to various questions about sets of points in the plane, including \Cref{quest:skewtris} (\Cref{def:trifree,def:skewcorners,def:biskew}). This gives \Cref{thm:main1,thm:main2}. We also give an example which shows that one cannot hope to prove strong upper bounds on $\Val(\Z_n)$ via a certain ``asymmetric" averaging argument (\Cref{prop:avgbad}).

The take-away of this paper is that STPP constructions yield subsets of $G \times G$ which satisfy dramatically stronger properties than that of being corner-free. While subsets satisfying these stronger properties do not imply STPP constructions in any obvious way, we believe that understanding them will be a stepping stone to understanding the power of the group--theoretic approach, and possibly towards improved upper bounds on $\omega$.
\section{Background}\label{sec:bg}
Bounds on $\omega$ from the group--theoretic approach are obtained by designing subsets of groups satisfying the following condition.
\begin{definition}
	A collection of triples of subsets $S_i, T_i, U_i$ of a group $G$ satisfy the simultaneous triple product property (or STPP for short) if
	\begin{enumerate}
		\item For each $i$, the sets $S_i, T_i, U_i$ satisfy the \emph{triple product property}: if $ss'^{-1}tt'^{-1}uu'^{-1}=I$ with $s, s' \in S_i, t,t' \in T_i, u,u' \in U_i$, then $s=s', t=t', u=u'$.
		\item Setting $S_i = A_iB_i^{-1}, T_j = B_jC_j^{-1}, U_k = C_kA_k^{-1}$, 
		\[ s_it_ju_k = I \iff i=j=k\]
		for all $s_i \in S_i, t_j \in T_j, u_k \in U_k$.
	\end{enumerate}
\end{definition}
The crucial fact is the following:
\begin{theorem}\label{prop:stppbound}\cite[Theorem 5.5]{cohn2005group}
	If $S_i, T_i, U_i \subseteq G$ satisfy the STPP, then
	\[\sum_i (|S_i||T_i||U_i|)^{\omega/3} \le \sum d_i^\omega \]
	where $d_i$'s are the dimensions of the irreducible representations of $G$.
\end{theorem}

The conditions of the STPP imply that the sets involved satisfy a simple ``packing bound" (see the discussion preceding \cite[Definition 2.3]{blasiak2017cap}).

\begin{proposition}\label{prop:pack}
If $S_i, T_i, U_i$ satisfy the STPP in a group $G$, then $\sum_i |S_i||T_i| \le |G|$, $\sum_i |T_i||U_i| \le |G|$, and $\sum_i |U_i||S_i| \le |G|$.
\end{proposition}

A particular type of STPP construction can be obtained from pairs of sets satisfying a condition termed the \emph{simultaneous double product property} in \cite{cohn2005group}.
\begin{definition}
	We say that sets $(A_i, B_i)_{i=1}^n$ satisfy the simultaneous double product property (or SDPP for short) if 
	\begin{enumerate}
		\item For all $i$, $aa'^{-1} = bb'^{-1}$ only has the solution $a=a', b=b'$ for $a,a' \in A_i, b,b' \in B_i$,
		\item $a_i (a_j')^{-1}b_j(b_k')^{-1} = 1$ implies $i=k$, where $a_i \in A_i, a_j' \in A_j, b_j \in B_j, b_k' \in B_k$.
	\end{enumerate}
\end{definition}

In \cite{cohn2005group} it was conjectured that one can achieve $\omega = 2$ using SDPP constructions in abelian groups. This amounts to the following.
\begin{conjecture}\cite[Conjecture 4.7]{cohn2005group}\label{conj:twofamilies}
	For arbitrarily large $n$, there exists an abelian group $G$ of order $n^{2-o(1)}$ and $n$ pairs of sets $A_i, B_i$ where $|A_i||B_i| > n^{2-o(1)}$ satisfy the SDPP.
\end{conjecture}

 If $G$ is a finite group, we let $X_G$ denote the tripartite 3-uniform hypergraph with vertex parts $X_1 = X_2 = X_3 = G$, and where $(g_1, g_2, g_3)$ is a hyperedge (a \emph{triangle}) whenever $g_1g_2g_3=I$. In the event that $G$ is nonabelian, it is important that we fix some ordering on the parts of $X_G$ here. Recall that a 3-uniform hypergraph is said to be \emph{linear} if any two vertices are contained in at most one hyperedge. For example, $X_G$ is linear. The matrix multiplication hypergraph $M_{p,q,r}$ is defined to be the supporting hypergraph of the matrix multiplication tensor; i.e.~it is the hypergraph with parts $[p]\times [q], [q] \times [r], [r] \times [p]$, and where $((i,j), (k,l), (m,n))$ is a hyperedge if and only if $j=k, l=m, n=i$. If $X$ is a hypergraph, we sometimes write $E(X)$ for the set of hyperedges of $X$.

It is convenient to view STPP constructions from a hypergraph perspective.

\begin{proposition}
There exist sets $S_i, T_i, U_i \subseteq G$, satisfying the STPP if and only if $X_G$ contains as an induced subhypergraph the disjoint union of $M_{|S_i|, |T_i|, |U_i|}$.
\end{proposition}
\begin{proof}
It follows from the first condition of the STPP that for all $i$, the subhypergraph induced by $A_{i} :=  S_iT_i^{-1}, B_{i} := T_iU_i^{-1} , C_{i} := U_iS_i^{-1}$ equals $M_{|S_i|, |T_i|, |U_i|}$. The second condition implies that $A_{i}$ and $A_{j}$ are disjoint when $i \neq j$, and similarly for the subsets of the other parts. The second condition also implies that the only hyperedges in the subhypergraph induced by $\sqcup_i A_{i}, \sqcup_i B_{i},\sqcup_i C_{i}$ are between sets of the form $A_{i}, B_{i}, C_{i}$, so the claim follows.

Conversely, suppose that $\sqcup_i A_{i}, \sqcup_i B_{i}, \sqcup_i C_{i}$ induce disjoint hypergraphs $M_{p_i, q_i, r_i}$. Fix some $i$, and for shorthand write $A:=A_{i}, B:=B_{i}, C:=C_{i}$ and let $p :=p_i, q:=q_i, r:=r_i$. Since $A,B,C$ induce $M_{p,q,r}$, we can by definition write $A =\{a_{ij}\}_{i \in [p],j \in [q]}, B =  \{b_{ij}\}_{i \in [q],j \in [r]}, C = \{c_{ij}\}_{i \in [r],j \in [p]} \in G$ where
\begin{equation}\label{eqn:tppgen}
a_{ij}b_{kl}c_{mn} = I \iff j=k, l=m, n=i.
\end{equation}
We claim that there exist $X = \{x_i\}_{i \in [p]}, Y = \{y_j\}_{j \in [q]}, Z = \{z_k\}_{k \in [r]}$ such that $a_{ij} = x_iy_j^{-1}$, $b_{jk} = y_jz_k^{-1}$, $c_{ki} = z_kx_i^{-1}$ for all $i \in [p], j \in [q], k \in [r]$. This can be accomplished by taking $x_0 = 1, x_i = a_{i0}a_{00}^{-1}$ for $i>0$, $y_i = a_{0i}^{-1}, z_i = c_{i0}$. Furthermore, \Cref{eqn:tppgen} implies that $X,Y,Z$ will satisfy the TPP. This shows that for each $i$ there are $X_i, Y_i, Z_i$ such that $A_{1,i} = X_iY_i^{-1}, A_{2,i} = Y_iZ_i^{-1}, A_{3,i} = Z_iX_i^{-1}$, and $X_i, Y_i, Z_i$ satisfy the TPP. The fact that they induce a disjoint union of hypergraphs implies that if $a \in A_{i}, b \in B_j, c \in C_k$, then $abc=I$ implies $i=j=k$, which implies the second condition in the definition of the STPP.
\end{proof}

\begin{remark}
The second direction of this proposition is essentially the fact that a complete 2-dimensional simplicial complex has trivial 1-cohomology with coefficients in any group.
\end{remark}

\subsection{Triangle Removal and the Group-Theoretic approach}\label{sec:trirem}

In \cite{serra2009combinatorial}, a nonabelian generalization of Green's arithmetic removal lemma \cite{green2005szemeredi} was shown to follow from the directed graph removal lemma  of Alon and Shapira \cite{alon2003testing}. Specifically, they showed the following:

\begin{theorem}\label{thm:groupremoval}
Let $G$ be a finite group of order $N$. Let $A_1, \ldots, A_m, m \ge 2$, be
sets of elements of $G$ and let $g$ be an arbitrary element of $G$. If the equation
$x_1x_2 \cdots x_m = g$ has $o(N^{m-1})$ solutions with $x_i \in A_i$, then there are subsets $A'_i \subseteq A_i$ with $|A_i \setminus A'_i| = o(N)$ such that there is no solution of the equation $x_1x_2 \cdots x_m = g$ with $x_i \in  A_i'$.
\end{theorem}
The best quantitative bounds for this theorem are due to Fox \cite{fox2011new}, and imply that if there are at most $\delta N^{m-1}$ solutions to $x_1 \cdots x_m = g$, one can remove subsets of $A_i$ of size $\varepsilon N$ and eliminate all solutions, when $\delta^{-1}$ is a tower of twos of height $O(\log \varepsilon^{-1})$.

\Cref{thm:groupremoval} implies the following.
\begin{corollary}\label{cor:packingbarrier}
If $X_i,Y_i,Z_i$ satisfy the STPP in a group $G$ of order $n$, then at least one of $\sum |X_i||Y_i|, \sum |X_i||Z_i|, \sum |Y_i||Z_i|$ is at most $o(n)$.
\end{corollary}
\begin{proof}
	Let $A_1 = \sqcup_i X_iY_i^{-1}, A_2 = \sqcup_i Y_iZ_i^{-1}, A_3 = \sqcup_i Z_iX_i^{-1}$. By definition of the STPP, the equation $x_1x_2x_3=I$ with $x_i \in A_i$ has $\sum_i |X_i||Y_i||Z_i|$ solutions. By the packing bound \Cref{prop:pack}, $\sum_i |X_i||Y_i| \le n, \sum_i |Y_i||Z_i| \le n, \sum_i |Z_i||X_i| \le n$, so by Cauchy--Schwarz there are at most $n^{3/2} = o(n^2)$ solutions to $a_1a_2a_3=I$. 
	
	Now suppose that $B_j \subseteq A_j$ satisfy $|B_j|/|A_j| > 0.9999$; we will show that there is a solution to $b_1b_2b_3=I$. For more than a $0.99$ fraction of the values of $i$ we must have $|B_1 \cap X_iY_i^{-1} |/|X_iY_i^{-1}| > 0.99$ (because $0.99 \cdot 1 + 0.01 \cdot 0.99 = 0.9999$) and similarly for the other sets. Hence by the pigeonhole principle there is some $i$ for which $|B_1 \cap  X_iY_i^{-1}|/|X_iY_i^{-1}| > 0.99, |B_2 \cap  Y_iZ_i^{-1}|/| Y_iZ_i^{-1}| > 0.99, |B_3 \cap  Z_iX_i^{-1}|/| Y_iZ_i^{-1}| > 0.99.$	Now consider the tripartite graph with parts $X_i, Y_i, Z_i$, where $(x,y)$ is an edge between $X_i$ and $Y_i$ if $xy^{-1} \in  B_1 \cap  X_iY_i^{-1}$, $(y,z)$ is an edge between $Y_I, Z_i$ when $yz^{-1} \in B_2 \cap  Y_iZ_i^{-1}$, and $(z,x)$ is an edge when $zx^{-1} \in B_3 \cap  Z_iX_i^{-1}$. Note that the existence of a triangle in this graph implies that there is a solution to $b_1b_2b_3=I$. First, note that at least $0.9|X_i|$ vertices in $X_i$ have at least $0.9|Y_i|$ neighbors in $Y_i$. (If this were not the case, there would be at most $0.9|X_i||Y_i| + 0.1 \cdot 0.9 \cdot |X_i||Y_i| \le 0.99|X_i||Y_i|$ edges between $X_i$ and $Y_i$, and hence $|B_1 \cap X_iY_i^{-1} |/|X_iY_i^{-1}| \le 0.99$, a contradiction.) Similarly, at least $0.9|X_i|$ vertices in $X_i$ have at least $0.9|Z_i|$ neighbors in $Z_i$. Hence at least $0.8|X_i|$ vertices in $X_i$ have $0.9|Y_i|$ neighbors in $Y_i$ and $0.9|Z_i|$ neighbors in $Z_i$. Pick any such vertex $x_0 \in X_i$. There must be an edge between a neighbor of $x_0$ in $Y_i$ and a neighbor of $x_0$ in $Z_i$, since if not, there would be at most $|Y_i||Z_i| - 0.9^2|Y_i||Z_i| = 0.19|Y_i||Z_i|$ edges between $Y_i$ and $Z_i$. Thus we have found our triangle.
	
	By \Cref{thm:groupremoval}, we can delete subsets of $A_i$ of size $o(n)$ to eliminate all solutions to $x_1x_2x_3=I$. On the other hand, any three subsets of the $A_i$'s of density $0.9999$ contain some such solution. Hence we must have $|A_i| = o(n)$ for some $i$.  \qedhere
\end{proof}

As a corollary of this proof, we have the following.

\begin{corollary}\label{cor:removaltight}
There exists an absolute constant $C > 1$ such that if $X_i,Y_i,Z_i$ satisfy the STPP in $\Z_q^\ell$, then at least one of $\sum |X_i||Y_i|, \sum |X_i||Z_i|, \sum |Y_i||Z_i|$ is at most $(q/C)^\ell$.
\end{corollary}
\begin{proof}
The proof of \Cref{cor:packingbarrier} shows that $A_1 = \sqcup_i X_iY_i^{-1}, A_2 = \sqcup_i Y_iZ_i^{-1}, A_3 = \sqcup_i Z_iX_i^{-1}$ have the following properties: there are at most $q^{3n/2}$ solutions to $a_1+a_2+a_3=0$, and any subsets of $A_1, A_2, A_3$ of density $0.9999$ each contain some such solution. At the same time, by \cite[Theorem 1]{fox2017tight}, if $A_1, A_2, A_3 \subseteq \Z_q^\ell$ and there are less than $\delta q^{2n}$ solutions to $a_1+a_2+a_3=0$, then we may remove $\varepsilon q^\ell$ elements from $A_1 \cup A_2 \cup A_3$ and eliminate all solutions, when $\delta = (\varepsilon/3)^{\Theta(\log q)}$.\footnote{While \cite[Theorem]{fox2017tight} is only stated for $\Z_p^\ell$, it extends to $\Z_q^\ell$ by the same argument via the use of \cite[Theorem A']{blasiak2017cap}.} In our setting, $\delta = q^{-n/2}$ and so $\varepsilon = 3 q^{\Theta(-n/\log q)} \le 3 C'^{-n}$ for some universal $C'$. Hence it must have been the case that one of $A_1, A_2, A_3$ had size at most $(q/C)^\ell$ to begin with, for some universal $C$. \qedhere
\end{proof}

One can interpret \Cref{{cor:packingbarrier}} as saying that the best upper bound on the rank of a direct sum of matrix multiplication tensors provable via the group--theoretic approach is superlinear. We remark the only important property of the matrix multiplication hypergraph for this result was that it satisfies a very weak ``regularity" condition. Specifically, considerations similar to those of the proof of \Cref{{cor:packingbarrier}} show the following:

\begin{theorem}\label{thm:quasiremoval}
Let $\varepsilon > 0$. Let $G$ be a group of order $n$. Let $X = \sqcup_{i=1}^3 A_i$ be a tripartite hypergraph with $o(n^2)$ triangles such that for any $Y_i \subseteq A_i$ with $|Y_i|/n \ge 1-\varepsilon$, there exists $y_i \in Y_i$ such that $(y_1, y_2, y_3) \in E(X)$. Then if $X$ is an induced subhypergraph of $X_G$, $|A_i| \le o(n)$ for $i=1,2,3$.
\end{theorem}

%\begin{proof}
%Suppose that $X$ is an induced subhypergraph of $X_G$. We identify the parts $A_i$ of $X$ with the corresponding subsets of $G$. Then since the number of triangles in $X$ is $o(n^2)$, the number of solutions to $a_1a_2a_3=I$ is $o(n^2)$. Furthermore, there 
%\end{proof}

%\begin{corollary}
%If $X_i,Y_i,Z_i$ satisfy the STPP in a group of order $n$, then at least one of $\sum |X_i||Y_i|, \sum |X_i||Z_i|, \sum |Y_i||Z_i|$ is at most $o(n)$.
%\end{corollary}
%\begin{proof}
%The STPP implies that the subhypergraphs of $X_G$ with parts $X_iY_i \subset X_1, Y_i^{-1}Z_i \subset X_2, Z_i^{-1}X_i \subset X_3$ are a matrix multiplication hypergraph $M_i$. Furthermore, the induced subhypergraph on the union of all such sets for all $i$ is a disjoint union of matrix multiplication hypergraphs. It therefore suffices to show that for some $\varepsilon > 0$, any three subsets of the parts of $M_{n,n,n}$ of density $\varepsilon$ induce a triangle. This is easily seen to hold for $\varepsilon = 1/3$.
%\end{proof}
\section{Equilateral trapezoid-freeness in hypergraphs and groups} \label{sec:hyper}

We begin with the observation that the matrix multiplication hypergraph is an extremal solution to a certain forbidden hypergraph problem.
\begin{proposition}\label{prop:extremalprob}
	Let $X$ be a linear tripartite hypergraph with parts of size $N$ such that any two vertices from different parts are incident to at most one common vertex in the third part. Then the number of triangles in $X$ is at most $N^{3/2}$. Furthermore, when $N$ is a square, an extremal example is the matrix multiplication hypergraph $M_{N^{1/2}}$.
\end{proposition}
The hypergraphs satisfying the condition of \Cref{prop:extremalprob} can be equivalently characterized as the linear hypergraphs that do not contain copies of the hypergraphs in \Cref{fig:forbid}. We remark that the proof of the upper bound in \Cref{prop:extremalprob} is closely related to the upper bound on the Tur\'an density of the 4-cycle.% We note that this fact can also be seen by using the extremal bound for $C_4$ and observing that the induced graph between two parts must be $C_4$-free. 

\begin{figure}
\begin{center}
\includegraphics[scale=1]{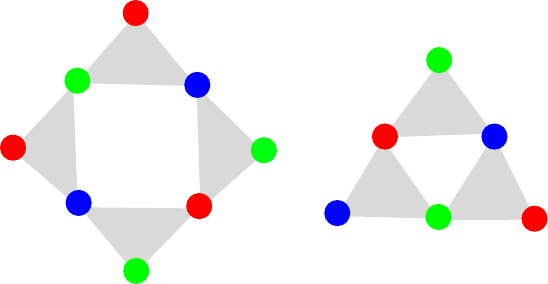}
\caption{The forbidden hypergraphs in \Cref{prop:extremalprob}, up to permutations of the three parts (represented by different colors).}
\label{fig:forbid}
\end{center}
\end{figure}

\begin{proof}
	Restricting our attention to one of the parts $X_1$ of $X$, let $d_v$ be the number of triangles that vertex $v \in X_1$ is contained in. Each $v \in X_1$ is contained in $d_v$ triangles, where the vertices of these triangles belonging to $X_2$ and $X_3$ are distinct (as $X$ is linear). Additionally, no pair of such vertices in $X_2$ and $X_3$ can be contained in a triangle incident to another vertex $u \in X_1$, so there are $2 \binom{d_v}{2}$ pairs of vertices in $X_2$ and $X_3$ that are contained in no common triangle. Let $(x_2, x_3)$ be some such pair of vertices. Observe that furthermore, for all $u \neq v \in X_1$, the set of vertices in $X_2$ and $X_3$ incident to the set of triangles containing $u$ cannot also contain both $x_2$ and $x_3$. For if this happened, there would be triangles $(v,x_2,x_3'), (v, x_2', x_3), (u, x_2, x_3''), (u, x_2'', x_3)$, and then $x_2$ and $x_3$ violate the constraint. The total number of triangles equals $m := \sum_{v \in X_1} d_v$, and by the prior observations it follows that   $\sum 2\binom{d_v}{2} +m \le N^2$. So $\sum d_v(d_v-1) + m = \sum d_v^2 \le N^2$. The conclusion follows from Cauchy--Schwarz.
	
	 To see that $M_{N^{1/2}}$ is extremal, note that it contains $N^{3/2}$ triangles, has parts of size $N$, and is linear. To see that it satisfies the second condition, let $(i,j)$ be a vertex in the first part, and let $(k,l)$ be a vertex in the second part. Then $(i,j)$ is contained in a common triangle with exactly the vertices in the third part of the form $(*,i)$, and $(k,l)$ is incident to exactly the vertices in the third part of the form $(l,*)$. Hence $(l,i)$ is the unique neighbor of both. The same argument shows the claim for vertices in any two parts.
\end{proof}

%\subsection{Equilateral trapezoid-freeness in groups}\label{sec:grouptrap}
The key definition in this paper is that of an ``equilateral trapezoid-free" triple of subsets of a group. The reason for this name will eventually be explained in \Cref{sec:zn}.

\begin{definition}\label{def:trap}
	Let $A, B, C \subseteq G$. We call $(A,B,C)$ equilateral trapezoid-free if the subhypergraph of $X_G$ induced by $A \subseteq X_1,B\subseteq X_2,C \subseteq X_3$ satisfies the conditions of \Cref{prop:extremalprob}. Equivalently, $(A,B,C)$ is equilateral trapezoid-free if for any fixed $a' \in A , b' \in B , c'\in C$, the following systems of equations in the variables $a \in A, b \in B, c \in C$ each have at most one solution:
	\begin{align*}
		I&=a'bc=ab'c,\\
		I&=a'bc=abc',\\
		I&=ab'c=abc'.
	\end{align*}
	Let $\Val(G)$ be the maximum number of solutions to $abc=I$ over all equilateral trapezoid-free triples $(A,B,C)$. 
\end{definition}

 The relevance of $\Val(G)$ to $\omega$ is due to the following.

\begin{proposition}\label{prop:stppval}
	Suppose that $X_G$ contains disjoint induced subhypergraphs $M_{n_i,m_i,p_i}$. Then, $\Val(G) \ge \sum_{i} n_im_ip_i$.
\end{proposition}
\begin{proof}
	By the same reasoning as in the second part of the proof of \Cref{prop:extremalprob}, $M_{n_i, m_i, p_i}$ satisfies the constraints of \Cref{def:trap} and contains $n_im_ip_i$ hyperedges. As the disjoint union of these hypergraphs satisfies these constraints as well, the claim follows.
\end{proof}
In fact, STPP constructions are essentially the only approach we know of for proving lower bounds on $\Val(G)$.

%We call such a triple of sets \emph{equilateral trapezoid free} for the following reason. The analogue condition in the integers is to find $A,B,C \subseteq [n]$ with at most one solution to the equations $n=a'+b+c=a+b'+c$, etc, and maximize the number of solutions to $a+b+c=n$. The points $\{(a,b,c) : a+b+c=n\}$ can be visualized as a triangular array of points inside of $\Z^2$. The sets $A,B,C$ correspond to sets of lines parallel to each of the three sides of this array. If they satisfy the condition of our question, then the subset of points in the array that are contained in three lines contain no points forming the vertices of an equilateral trapezoid (with sides aligned with the sides of the array). Here an equilateral triangle is considered as a special case as an equilateral trapezoid.

To start, we have the following trivial bounds.

\begin{proposition}\label{prop:trivialvalbds}
	For any group $G$, $|G| \le \Val(G) \le |G|^{3/2}$.
\end{proposition}
\begin{proof}
	The lower bound is obtained by the triple $(\{I\}, G, G)$. The upper bound follows from \Cref{prop:extremalprob}.
\end{proof}

The following super-multiplicative behavior of $\Val$ is easily checked.
\begin{proposition}\label{prop:tensorval}
If $(A,B,C)$ is equilateral trapezoid-free in $G$, and $(A',B',C')$ is equilateral trapezoid-free in $H$, then $(A \times A', B \times B', C \times C')$ is a equilateral trapezoid-free in $G \times H$.
\end{proposition}
It is also easily seen that being equilateral trapezoid-free is preserved by cyclic permutations of the three sets.
\begin{proposition}\label{prop:symval}
If $(A,B,C)$ is equilateral trapezoid-free, then so is $(B,C,A)$.
\end{proposition}

By an application of \Cref{thm:quasiremoval} combined with the observation that near-extremal solutions to \Cref{prop:extremalprob} are highly ``regular", we have the following weak improvement to the trivial upper bound of $|G|^{3/2}$.

\begin{proposition}\label{prop:val2}
	For any group $G$, $\Val(G) \le o(|G|^{3/2})$.
\end{proposition}
\begin{proof}
	Suppose for contradiction that there exists $\varepsilon_0 > 0$ such that $\Val(G) > \varepsilon_0 |G|^{3/2}$, and let $A_0,B_0,C_0 \subseteq G$ witness $\Val(G) = \varepsilon_0 |G|^{3/2}$. Next consider the triple $(A,B,C) := (A_0 \times B_0 \times C_0, B_0 \times C_0 \times A_0, C_0 \times A_0 \times B_0)$, which is equilateral-trapezoid free inside of $H:=G^3$ by \Cref{prop:tensorval} and \Cref{prop:symval}, and witnesses $\Val(H) \ge \varepsilon |H|^{3/2}$ where $\varepsilon := \varepsilon_0^3$. Let $|H| = N$. Let $X$ be the tripartite hypergraph with parts $A,B,C$ and where there is a triangle between all triples $(a,b,c)$ where $abc=I$. Let $n:=|A| = |B| = |C|$. By \Cref{prop:extremalprob} we must have $n \ge \varepsilon^{2/3} N $. Note that the number of triangles in $X$ equals  $\varepsilon N^{3/2} \ge \varepsilon n^{3/2}$. In what follows, the degree of a vertex in $X$ refers to the number of triangles containing it.
	
	Let $Y$ be the random variable that is uniformly distributed over the multiset of vertex degrees from one part of $X$, say $A$. Then $\E[Y] \ge \varepsilon n^{1/2}$ and $\E[Y^2] \le n$ (this second inequality follows from the use of Cauchy--Schwarz in the proof of \Cref{prop:extremalprob}). By the Payley-Zygmund inequality, for any $\theta>0$, $\Pr(Y > \theta \cdot \varepsilon n^{1/2} ) \ge (1-\theta^2)\varepsilon^2$. Taking $\theta = 1/2$, we conclude that at least $p \cdot n := 3n\varepsilon^2/4$ vertices in $A$ have degree at least $\varepsilon n^{1/2} /2$. This holds for $B$ and $C$ as well.
	
	 Now let $S, T,$ and $U$ be any subsets of $A,B,C$ of size at least $n(1-p/\lambda)$; we'll pick $\lambda \in \N$ later. Then the number of triangles incident to any one of these sets, say $S$, is at least
	\[np(1-\lambda^{-1}) \cdot \varepsilon n^{1/2} /2 = (3/8) n^{3/2} \varepsilon^3 (1-\lambda^{-1}),\]
	and the number of triangles incident to $[n] \setminus T$ or $[n] \setminus U$, sets of size at most $np/\lambda$, is at most 
	\[(n^2 \cdot np/\lambda)^{1/2} = (3^{1/2}/2) n^{3/2} \varepsilon \lambda^{-1/2} \]
	by Cauchy--Schwarz. It follows that the number of triangles with one vertex in each of $S,T,U$ is at least
	\[(3/8) n^{3/2} \varepsilon^3 (1-\lambda^{-1}) - 2 \cdot (3^{1/2}/2) n^{3/2} \varepsilon \lambda^{-1/2} \]
	which is greater than 1 for $\lambda \gg \varepsilon^{-4}$. In summary, between any three subsets of $A,B,C$ size roughly $n(1-\varepsilon^6)$, there is a triangle.
	
	Recall that $n \ge \varepsilon^{2/3} N$. Since $X$ has at most $N^{3/2} \le o(N^2)$ triangles, by \Cref{thm:quasiremoval} we can remove $o(N) = o(n)$ vertices to remove all triangles. But by what we have just shown, after deleting this few vertices some triangle will remain, a contradiction.
\end{proof}

\begin{remark}
By combining this proof with \cite{fox2017tight}, it follows that for fixed $n$ and some $\varepsilon > 0$, $\Val(\Z_n^\ell) \le O(n^{3/2(1-\varepsilon)\ell})$.
\end{remark}

\section{$\Val(\Z_n)$ and its applications}\label{sec:zn}
Our weakest conjecture is the following.

\begin{conjecture}\label{conj:val}
	For all $\varepsilon > 0$, $\Val(\Z_n) \le O(n^{1+\varepsilon})$.
\end{conjecture}
In this section we give our potential applications of this conjecture. We then introduce several related quantities and make preliminary progress on understanding them. 

While the quantity $\Val(\Z_n)$ may seem opaque from \Cref{def:trap}, it can easily be visualized. This is done by first considering the natural notion of an equilateral trapezoid-free subset of the plane, which is convenient to introduce sooner rather than later. Throughout this section, we let $\Delta_{n+1} = \{(a,b,c) \in \Z_{\ge 0}^3: a+b+c=n\}$. A subset of $\Delta_{n+1}$ is said to be corner-free if it contains no configuration $(x+\delta, y, z), (x,y+\delta,z), (x,y,z+\delta)$.

\begin{definition}\label{def:trapplane}
	Let $A, B, C \subseteq \{0, \ldots, n\}$. We call $(A,B,C)$ an equilateral trapezoid-free triple if for any fixed $a', b', c'$, the following systems of equations in the variables $a \in A, b \in B, c \in C$ each have at most one solution:
	\begin{align*}
		n&=a'+b+c=a+b'+c\\
		n&=a'+b+c=a+b+c'\\
		n&=a+b'+c=a+b+c'.
	\end{align*}
	Let $\Val(n)$ be the maximum number of solutions to $a+b+c=n$ over all equilateral trapezoid-free triples $(A,B,C)$. 
\end{definition}

We may visualize equilateral trapezoid-free sets as follows. Draw $\Delta_{n+1}$ in the plane as a triangular grid of points. Sets $A,B,C$ correspond to collections of lines parallel to the sides of $\Delta_{n+1}$, and a solution $a+b+c=n$ corresponds a point in $\Delta_{n+1}$ contained in one line in each of these three directions. Let $S \subseteq \Delta_{n+1}$ be the collection of all such points. A violation of a constraint of \Cref{def:trapplane} corresponds to either a subset of 3 points in $S$ forming an equilateral triangle with sides parallel to the sides of $\Delta_{n+1}$, or a subset of 4 points with sides parallel to the sides of $\Delta_{n+1}$ forming an equilateral trapezoid. Equivalently, we are deleting lines parallel to the sides of $\Delta_{n+1}$ to eliminate all of such configurations, while leaving as many points as possible. The maximum possible number of points left equals $\Val(n)$. See \Cref{fig:traps}.

\begin{figure}
	\begin{center}
		\includegraphics[scale=0.5]{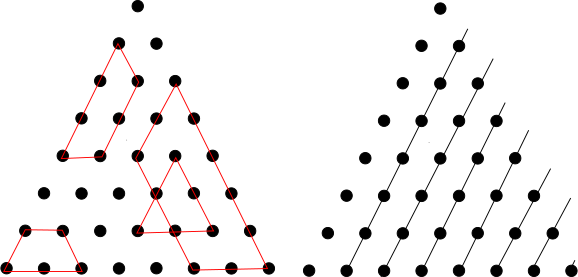}
		\caption{Left: some forbidden trapezoids and triangles in $\Delta_8$. Right: a trapezoid-free subset of $\Delta_8$ of size $8$ obtained by deleting all lines but one along one direction.}\label{fig:traps}
	\end{center}
\end{figure}

The following shows that $\Val(n)$ and $\Val(\Z_n)$ are essentially the same.
\begin{proposition}\label{prop:resasonable}
	\begin{enumerate}
		\item $\Val(n) \ge \Val(n-1)$.
		\item $1+2 \cdot \Val(2n) \ge \Val(\Z_n) \ge \Val(\lfloor n/3 \rfloor)$.
		\item For $n \ge 6n'$, $\Val(\Z_n) \ge \Val(\Z_n')/2-1$
	\end{enumerate}

\end{proposition}
\begin{proof}
	Suppose that $\Val(n)$ is witnessed by sets $A,B,C$. For $N > n$, $A+(N-n), B, C$ then witness $\Val(N) \ge \Val(n)$, which shows (1). If we take $N = 3n$, we have that $A+2n \subseteq \{0,\ldots, N\}$ and $B, C \subseteq \{0,\ldots, N/3\}$, so $a+b+c \le 5N/3 < 2N$. Since $a+2n + b+c = 0 \bmod N \iff a+2n+b+c =N$, this implies that the sets $A+2n, B, C$ are equilateral trapezoid-free when viewed as subsets of $\Z_N$. This shows one direction of (2). In the other direction, suppose $\Val(\Z_n)$ is witnessed by $A,B,C \bmod n$. There are at least $(\Val(\Z_n)-1)/2$ solutions to one of $a+b+c=n, a+b+c=2n$; let $N$ be the right-hand side of the most frequently satisfied equation. Since every solution to $a+b+c=N$ is a solution to $a+b+c = 0 \bmod n$, $A,B,C$ must be equilateral trapezoid-free when viewed as subsets of $\{0, \ldots, N\}$. This shows the other direction of (2).
	
	Finally, (3) follows from (1) and (2).
\end{proof}

\begin{theorem}\label{thm:polyval}
	Suppose that one can achieve $\omega = 2$ via STPP constructions in the family of groups $\Z_q^\ell$, $q$ a prime power. Then there exists a constant $c>0$ such that $\Val(\Z_n) \ge \Omega(n^{1+c}).$
\end{theorem}
\begin{proof}
By \Cref{cor:packingbarrier} and \Cref{cor:removaltight}, any STPP construction with sets $X_i, Y_i, Z_i$ satisfies $\sum |X_i||Y_i| \le (q/C)^\ell$ (we choose the $X$ and $Y$ sets without loss of generality) where $C$ is an absolute constant. By H\"older's inequality, $\sum (|X_i||Y_i||Z_i|)^{2/3} \le q^{2\ell/3}(q/C)^{\ell/3} = (q/C^{1/3})^\ell$. 
If we can obtain $\omega < 3 - \alpha$ via \Cref{prop:stppbound}, then
\begin{align*}
	q^\ell < \sum (|X_i||Y_i||Z_i|)^{2/3 \cdot \alpha + (1-\alpha)} &=	\sum (|X_i||Y_i||Z_i|)^{2/3 \cdot \alpha} (|X_i||Y_i||Z_i|)^{1-\alpha}\\
	&\le (\sum (|X_i||Y_i||Z_i|)^{2/3})^\alpha (\sum |X_i||Y_i||Z_i|)^{1-\alpha}\\
	&\le  (q/C^{1/3})^{\alpha \ell} \Val(G)^{1-\alpha}
\end{align*}
so $\Val(G) > q^\ell (C^{\alpha/3(1-\alpha)})^\ell$. By choosing $\alpha$ sufficiently close to 1, $\Val(G) > q^\ell 4^\ell$. By taking $k$-fold products of the sets defining the STPP constructions (using that products of STPPs are STPPs \cite[Lemma 5.4]{cohn2005group}), we find that $\Val(\Z_q^{k\ell}) > (4q)^{k\ell}$ for all $k$.

Let $N = k\ell$. Consider the embedding $\varphi : \Z_q^N \to \Z_{(3q)^N}$ defined by $\varphi(x_1, \ldots, x_N) = x_1 + x_2  3q + \cdots + x_n (3q)^{N-1}$. Since $\sum y_i (3q)^{i-1}$ has a unique such expression in $\Z_{(3q)^N}$ when $y_i < 3q$, it follows that that
\[a_1+a_2+a_3 \neq a_4+a_5+a_6 \implies \varphi(a_1) + \varphi(a_2) + \varphi(a_3) \neq \varphi(a_4) + \varphi(a_5) + \varphi(a_6).\]
Hence the image of an STPP under $\varphi$ is an STPP inside of $\Z_{(3q)^N}$, so $\Val( \Z_{(3q)^N}) > (4q)^N$. Because this holds for some particular $q$ and all $N = k\ell$, by part (3) of \Cref{prop:resasonable} the theorem follows.
\end{proof}

\begin{corollary}
Suppose that there is a family of STPP constructions obtaining $\omega = 2$ in a family of abelian groups with a bounded number of direct factors. Then there exists a constant $c > 0$ such that $\Val(\Z_n) \ge \Omega(n^{1+c})$.
\end{corollary}
\begin{proof}
Suppose we have a family of STPP construction in groups of the form $G = \Z_{m_1}\times \cdots \times \Z_{m_\ell}$, with $\ell$ fixed. We can then obtain an STPP construction in $\Z_p^\ell$, where $p$ is the smallest prime greater than $\max_{i \in k} 3m_i$, by taking the image of this STPP under the map sending $(x_1, \ldots, x_k) \to x_1+x_2p+ \cdots + x_kp^{k-1}$. As $k$ is fixed, it follows from Bertrand's postulate that $p^k \le O(|G|)$. The inequality \Cref{prop:stppbound} then implies that one can also obtain $\omega = 2$ in the family of groups $\Z_p^\ell$, so we conclude by \Cref{thm:polyval}.
\end{proof}
\begin{remark}
Although we expect that \Cref{thm:polyval} is true when the hypothesis is extended to arbitrary abelian groups, we do not know how to generalize to e.g.~$\Z_n^\ell$ for arbitrary $n$. This is due to the fact that better bounds on the size of 3-matchings in cyclic groups with prime power modului are known than for general moduli (compare Theorems A and A' in \cite{blasiak2017cap}). To the best of our knowledge, it is an open problem whether the known bounds for non-prime power moduli are tight. For prime power moduli, the known bounds are tight by \cite{kleinberg2016growth}.
\end{remark}
Next we show that sufficiently strong simultaneous double product property constructions, which are known to prove $\omega < 2.48$ \cite[Proposition 4.5]{cohn2005group}, imply strong lower bounds on $\Val(\Z_n)$. We thank Chris Umans for informing us of the fact that if \Cref{conj:twofamilies} is true, then it is true in cyclic groups, which motivated the following theorem.

\begin{theorem}\label{thm:twofam43}
If \Cref{conj:twofamilies} is true, then for any $\varepsilon > 0$, $\Val(\Z_n) \ge O(n^{4/3-\varepsilon})$.
\end{theorem}
\begin{proof}
We begin by recalling how to turn an SDPP construction into an STPP construction \cite[Section 6.2]{cohn2005group}. Let $S \subset \Delta_n$ be corner-free and of size $n^{2-o(1)}$. For all $v = (v_1,v_2,v_3) \in S$, define the following subsets of $G^3$:
\begin{align*}
A_v &= A_{v_1} \times \{1\} \times B_{v_3},\\
B_v &= B_{b_1} \times A_{v_2} \times \{1\},\\
C_v &= \{1\} \times B_{v_2} \times A_{v_3}.
\end{align*}
It can be verified that the sets $(A_v, B_v, C_v)_{v \in S}$ satisfy the STPP. Hence \Cref{conj:twofamilies} yields an STPP with $n^{2-o(1)}$ triples of sets of size $n^{2-o(1)}$, inside a group of size $n^{6 - o(1)}$.

 Now consider the map from $G^3 = \Z_{m_1} \times \cdots \times \Z_{m_k}$, where $m_1 \le m_2 \le \cdots \le m_k$, to $G' := \Z_{\prod_i 3m_i}$ sending $(x_1, \ldots, x_k)$ to $x_1+(3m_1)x_2 + (3m_1)(3m_2)x_3 + \cdots $. First, the image of sets satisfying the STPP under this map still satisfy the STPP. This shows that $\Val(G') > n^{2-o(1)} \cdot n^{3(2-o(1))} = n^{8-o(1)}$. Second, for all fixed $c>0$ and $\ell \in \N$, $G^3$ cannot contain a subgroup of size $|G^3|^c$ generated by elements of order at most $\ell$ by \cite[Proposition 4.2]{blasiak2017cap}. Hence the number of $m_i$'s which are at most $\ell$ is at most $\log_2(|G^3|^c)$. The number of $m_i$'s which are greater than $\ell$ is trivially less than $\log_\ell |G^3|$. So,
 \[|G'| = \prod_{m_i \le \ell} 3m_i \prod_{m_i> \ell} 3m_i \le 3^{\log_2(|G^3|^c) + \log_\ell |G^3|} \cdot |G^3|.\]
 By taking $c$ sufficiently small and $\ell$ sufficiently large, this is at most $n^{6+\delta}$ for any desired $\delta> 0$. The claimed bound follows.
\end{proof}
Note that here there is no restriction on the family of abelian groups in consideration, unlike there was in the previous theorem.

\subsection{Relaxations of $\Val(\Z_n)$}
In this section we explore some strengthenings of \Cref{conj:val} which may be easier to understand. We start by discussing an over-strengthening of \Cref{conj:val} which \emph{cannot} give any barriers. We then discuss a few strengthenings for which our knowledge is embarrassingly bad, including the notions of skew-corner free sets from the introduction.

Considerations of the proof of the $n^{3/2}$ upper bound of \Cref{prop:trivialvalbds} reveal that it actually held for a (possibly) much weaker problem, where one only requires that the expected number of solutions of one of the three systems of two equations in \Cref{def:trap} is at most 1. We begin by noting that this upper bound is essentially best-possible for this weakened problem. In other words, one cannot hope to prove \Cref{conj:val} via an ``asymmetric" averaging argument.

\begin{proposition}\label{prop:avgbad}
There exist $A, B, C \subseteq \Z_n$ such that 
\[\E_{a' \in A, b' \in B}\left [\#\{(a,b,c): 0=a'+b+c=a+b'+c\}\right ] \le 1\]
and there are $n^{3/2 - o(1)}$ solutions to the equation $a+b+c=0$ with $a \in A, b \in B, c \in C$.
\end{proposition}

%By Cauchy--Schwarz, $\sum_{c \in C} r(A,B,c) \le \sqrt{\sum_{c \in C} r(A,B,c)^2} \sqrt{|C|} = \sqrt{|A| |B| |C|} \le p^{3/2}$. How close can we come to achieving this?

\begin{proof}
	Let $r(A,B,c)$ denote the number of representations of $c$ as $a+b$. First note that the proposition is equivalent to the statement that  $\sum_{c \in C} r(A,B,-c)^2 \le |A| |B|$ and $\sum_{c \in C} r(A,B,-c) = n^{3/2 - o(1)}$.
	
	Let $S \subset [n]$ be 3AP-free and of size $n^{1-o(1)}$. Consider the sets
	\[A = B= [3n^2,4n^2] \cup \bigcup_{x \in S} [x n, xn + n/2],C= -\{2xn+y : x \in S, y \in [n] \}\]
	regarded as subsets of $\Z_{100n^2}$.
	By definition, for any $x \in S$ and $y \in [n]$, $-(2xn+y) =c \in C$. If we have any representation $-c = a + b$, then $a, b < 3n^2$. So we have $a = x_1n+y_1, b = x_2n+y_1$ with $x_1, x_2 \in S$ and $1 \le y_1, y_2 \le n$. So $(x_1+x_2)n + (y_1+y_2) = 2xn + y$, and then we are forced to have $x_1 + x_2 = 2x$ and $y_1+y_2 = y$. But because $S$ is 3AP-free, we must have $x_1 = x_2 = x$. Hence $r(A,B,-c)$ is exactly the number of solutions to $y = y_1 + y_2$ with $y_1, y_2 \in [n]$, which is $\Omega(n)$ for $\Omega(n)$ choices of $y \in [n]$. Hence $\sum_{c \in C} r(A,B,-c) = \Theta(|S| n^2) = n^{3-o(1)}$. Also, we have that $\sum_{c \in C} r(A,B,-c)^2 =n^{4-o(1)} < |A| |B| = \Theta(n^4)$, and we are done.
\end{proof}
Can one find a construction achieving $n^{3/2 - o(1)}$ for the averaging version of \Cref{def:trap} that involves all three systems of equations? That is:
\begin{question}
What is the maximum over all $A, B, C \subseteq \Z_n$ satisfying 
\begin{align*}
\E_{a' \in A, b' \in B}\left [\#\{(a,b,c): 0=a'+b+c=a+b'+c\}\right ] &\le 1,\\
\E_{a' \in A, c' \in C}\left [\#\{(a,b,c): 0=a'+b+c=a+b+c'\}\right ] &\le 1,\\
\E_{b' \in B, c' \in C}\left [\#\{(a,b,c): 0=a+b'+c=a+b+c'\}\right ] &\le 1,
\end{align*}
of the number of solutions to $a+b+c=0$?
\end{question}

There are a number of relaxations of the quantity $\Val(n)$ for which we know basically nothing. A first relaxation that still seems very stringent is that of a \emph{triforce-free} triple, defined as follows.
\begin{definition}\label{def:trifree}
	Let $A, B, C \subseteq \{0, \ldots, n\}$. We say that $(A,B,C)$ is triforce-free if there is no solution to
	\[a+b+c'=a+b'+c=a'+b+c=n\]
	with $a \neq a', b \neq b', c \neq c'$.
	We write $\Val(\triforce{0.4cm},n)$ for the maximum over all such $A,B,C$ of the number of solutions to $a+b+c=n$.
\end{definition}

This condition just says that $\{(a,b,c) \in A \times B \times C: a+b+c=n\} \subseteq \Delta_{n+1}$ is corner-free. Equivalently, $(A,B,C)$ is triforce-free if the hypergraph with parts $A,B,C$ and triangles between any triples summing to $n$ does not contain the triforce hypergraph (the second hypergraph in \Cref{fig:forbid}). As every equilateral trapezoid-free triple of sets also has this property, we have the following.

\begin{proposition}
$\Val(\triforce{0.4cm},n) \ge \Val(n)$.
\end{proposition}

Here is an even weaker notion than that of being triforce-free. We thank Ryan O'Donnell for suggesting this definition.

\begin{definition}\label{def:skewcorners}
	We call $S \subseteq \Delta_n$ skew-corner free if for $(a,b,c), (a,b',c') \in S$, it holds that $(a+b - b', b'', c'') \notin S$ for all $b'', c''$, and this remains true after any permutation of the coordinates of $S$. 
\end{definition}
Pictorially, this says that for any two points lying on an axis-aligned line in $\Delta_n$, the parallel line passing through a third point that would form a corner with these two points must contain no points. As \Cref{def:trifree} yields corner-free subsets of $\Delta_n$ obtained by deleting axis-aligned lines, it follows that this is a relaxation of being triforce-free. More formally, we have the following.

\begin{proposition}\label{prop:tridel}
The largest skew-corner free subset of $\Delta_{n+1}$ is at least $\Val(\triforce{0.4cm},n)$.
\end{proposition}
\begin{proof}
Suppose that $A,B,C \subset \{0, \ldots, n\}$ satisfy the conditions of \Cref{def:trifree}, and let $S = \{(a,b,c) \subset A \times B \times C: a+b+c=n\} \subseteq \Delta_{n+1}$. Suppose for contradiction that $(a,b,c), (a,b',c') \in S$ and $(a-b+b',b'',c'') \in S$. Since $a-b+b' \in A, b \in B, c' \in C$ and $(a-b+b')+b+c' = a+b'+c' = n$, it follows that $(a-b+b', b,c') \in S$. But this is impossible: the three solutions $a+b'+c' = n, a+b+c=n,(a-b+b') +  b +c'$ violate \Cref{def:trifree}. One reasons similarly about other permutations of coordinates.
\end{proof}

The best lower bound that we know on the size of the largest skew-corner free subset of $\Delta_n$ is $\Omega(n)$; $n$ is obtained trivially by taking one line on the side of $\Delta_n$, and it is not hard to improve this to $3n/2$. We have found examples exceeding these bounds with computer search (see \Cref{fig:symm}).

If we weaken \Cref{def:skewcorners} by dropping the requirement that the condition holds for all permutations of coordinates, we are led to the following notion.
\begin{definition}\label{def:biskew}
We say $S \subset [n]^2$ is skew corner-free if it contains no configuration $(x,y), (x,y+d), (x+d, y')$ with $d \neq 0$.
\end{definition}
\begin{proposition}\label{prop:skewdel}
The largest skew corner-free subset of $[n]^2$ is at least as big as the largest skew corner-free subset of $\Delta_n$.
\end{proposition}
\begin{proof}
Given a skew corner-free set $S \subseteq \Delta_n$, let $S'$ be its projection onto the first two coordinates. This is a subset of $\{0,\ldots, n-1\}^2$ of size $|S|$. By definition, it contains no points $(a,b), (a,b'), (a+b-b',b'')$. By shifting each point by $(1,1)$ we obtain a subset of $[n]^2$ with this property.
\end{proof}

As a consequence, we have \Cref{thm:main1,thm:main2}.
\begin{proof}[Proof of \Cref{thm:main1} and \Cref{thm:main2}]
By \Cref{thm:polyval}, if $\omega = 2$ via STPP constructions in $\Z_q^\ell$, then $\Val(\Z_n) \ge \Omega(n^{1+c})$. By \Cref{prop:resasonable}, $\Val(\Z_n) = \Theta(\Val(n))$,  and by \Cref{prop:skewdel,prop:tridel}, $\Val(n)$ is at most the size of the largest skew corner-free subset of $[n]^2$. This proves \Cref{thm:main1}. One similarly concludes \Cref{thm:main2} by using \Cref{thm:twofam43}.
\end{proof}

We have the following nontrivial lower bound for this relaxed problem, due to a MathOverflow answer of Fedor Petrov \cite{petrovanswer}.  

\begin{proposition}\label{prop:nontrivgrid}
There is a skew corner-free subset of $[n]^2$ of size $\Omega(n \log n/\sqrt{\log \log n})$.
\end{proposition}
\begin{proof}
$A \subseteq [n]$ is called \emph{primitive} if for all $a \neq a' \in A, a \nmid a'$. It is easily seen that if $A$ is primitive then the set of points $(a,ka) \subseteq [n]^2$ for all $k\le n/a$ avoids the forbidden configurations. This gives a subset of size $n \sum_{a \in A} 1/a$. At the same time, there exists a $c>0$ and a primitive set $A$ where $\sum_{a \in A} 1/a > c \log n/(\log \log n)^{1/2}$ \cite{erdos1967theorem}. We note that this is best-possible, matching (up to the constant) an upper bound on $\sum_{a \in A} 1/a$ for primitive $A$ due to Behrend \cite{behrend1935sequences}.
\end{proof}

This construction breaks when we strengthen the definition of skew corner-freeness in $[n]^2$ to forbid skew corners with two points parallel to the $x$ axis. This corresponds to the following notion.

\begin{definition}\label{quest:twodir}
We say $S \subset [n]^2$ is bi-skew corner-free if it contains no configurations $(x,y), (x,y+d), (x+d, y')$ or $(x,y), (x+d, y), (x',y+d)$, with $d \neq 0$.
\end{definition}

As far as we know, it is possible that the largest bi-skew corner-free subset has size $O(n)$.
\section{Acknowledgments}
I thank Ryan O'Donnell for many useful discussions about these problems, and for suggesting \Cref{def:skewcorners}. I also thank Chris Umans for making comments which motivated this paper early on, and in particular, which motivated \Cref{thm:twofam43}.
        \bibliographystyle{amsalpha}
\bibliography{refs}
\end{document}